\documentclass[11pt]{article}
\usepackage[utf8]{inputenc}
\usepackage[T1]{fontenc}
\usepackage{amsmath,amssymb,amsthm}
\usepackage{mathrsfs}
\usepackage{hyperref}
\usepackage[margin=1in]{geometry}

\newtheorem{theorem}{Theorem}[section]
\newtheorem{lemma}[theorem]{Lemma}
\newtheorem{proposition}[theorem]{Proposition}
\newtheorem{corollary}[theorem]{Corollary}
\theoremstyle{definition}

\theoremstyle{remark}

\title{A Hardy-space operator acting on $\zeta$-zeros}
\author{Juan Manzur and Waleed Noor}
\date{}

\begin{document}

\maketitle

\begin{abstract}
A Hardy space approach to the Nyman-Beurling and Báez-Duarte criterion for the Riemann Hypothesis (RH) was introduced in \cite{Noor19}. It states that the RH holds if and only if the span of a particular sequence of functions $(h_k)_{k\geq2}$, denoted $\mathcal{N}$, is dense in the Hardy space $H^2$. In this note we construct a bounded operator $T$ on \(H^{2}\) such that if the closure of $\mathcal{N}$ is $T$-invariant, then either the RH holds or there exists an infinite progression of $\zeta$-zeros. In particular, this yields an operator-theoretic criterion for the Riemann hypothesis as an invariant subspace question.
\end{abstract}

\section{Introduction}
The Hardy space \(H^{2}\) of analytic functions on the open unit disk is identified with \(\ell^{2}(\mathbb{N}_{0})\) by
\[
f(z) = \sum_{j = 0}^{\infty}a_{j}z^{j}\longleftrightarrow (a_{j})_{j\geq 0},\qquad \| f\|_{H^{2}}^{2} = \sum_{j = 0}^{\infty}|a_{j}|^{2}.
\]
The $H^2$-functions defined by
\[
h_{k}(z) = \frac{1}{1 - z}\log \left(\frac{1 + z + \cdots + z^{k - 1}}{k}\right),\qquad k\geq 2
\]
and their linear span denoted \(\mathcal{N}\) are associated with an $H^2$-reformulation of the Riemann hypothesis (RH) as follows (see Noor \cite{Noor19}):
\begin{equation}
The \ RH \ holds \ \iff \mathcal{N} \ is \ dense \ in \ H^{2}.\tag{A}
\end{equation}
For \(s \in \mathbb{C}\) with \(\Re s > \frac{1}{2}\), let \(k_{s} \) denote the Riesz representative of the bounded functional \(\Lambda^{(s)}\) on $H^2$ introduced in \cite[p. 6]{GKNS24}. Explicitly, they have Taylor series representations
\begin{equation}
k_{s}(z) = -\frac{1}{\bar{s}} -\frac{1}{\bar{s}}\sum_{j = 1}^{\infty}\big((j + 1)^{1 - \bar{s}} - j^{1 - \bar{s}}\big)z^{j}. \tag{B}
\end{equation}
Since
$
(j + 1)^{1 - \bar{s}} - j^{1 - \bar{s}} = O_{s}(j^{- \Re s}),
$
 the coefficient sequence in (B) belongs to \(\ell^{2}\) whenever \(\Re s > \frac{1}{2}\). Thus \(k_{s} \in H^{2}\) throughout the half-plane \(\Re s > \frac{1}{2}\). These functionals \(\Lambda^{(s)}\) satisfy
\begin{equation}
\Lambda^{(s)}(h_k)=\langle h_k,k_s\rangle = -\frac{\zeta(s)}{s}\big(k^{1 - s} - 1\big),\qquad \Re s > \frac{1}{2}. \tag{C}
\end{equation}
This corresponds to the formula in \cite[p. 6]{GKNS24} after removing the factor \(1 / k\) used there in the definition of \(h_{k}\). See \cite[Appendix]{CMSN24} for a direct proof of (C). It follows that for any \(\Re s > \frac{1}{2}\)
\begin{equation}
k_{s}\in \mathcal{N}^{\perp}\iff \zeta (s) = 0. \tag{D}
\end{equation}
The motivation for this note is the following $\mathbf{question}$: \emph{Does there exist a bounded operator T on $H^2$ whose adjoint $T^*$ maps the kernel $k_s$ to a scalar multiple of a translated kernel $k_{s+\alpha}$ with $\Re \alpha > 0$?}\\ \\
If such a $T$ exists with the additional property that $\overline{\mathcal{N}}$ is a $T$-invariant subspace (so $T^*\mathcal{N}^{\perp}\subset \mathcal{N}^{\perp}$), it would follow by (D) that $k_{\rho+n\alpha}\in \mathcal{N}^{\perp}$ for all $n\geq 1$ whenever $\zeta (\rho) = 0$ for some $\rho$ with $\Re \rho>1/2$. This implies that $\zeta$ has an infinite progression of zeros $\rho, \rho+\alpha, \rho+2\alpha, \ldots$, which is impossible by classical $\zeta$-zero density results and the fact that $\zeta(s)\neq 0$ for $\Re s>1$. Such a contradiction would imply the RH. In this note we show that there exist operators that answer the question above.

\section{Definition of the operator}

Write \(e_{j}(z) = z^{j}\) and fix \(\tau \in \mathbb{R}\). Define the diagonal operator \(D_{\tau}\) on $H^2$ by
\begin{equation}
D_{\tau}e_{j} = (j + 1)^{i\tau}e_{j},\qquad j\geq 0. \tag{1}
\end{equation}
Since \(|(j + 1)^{i\tau}| = 1\), \(D_{\tau}\) is unitary and
\begin{equation}
\| D_{\tau}\|_{H^2} = 1,\qquad D_{\tau}^{*} = D_{-\tau}. \tag{2}
\end{equation}
We also use the \emph{Copson operator} \(\mathcal{C}\) (see \cite{Copson}) defined on $\ell^2$ (and naturally on $H^2$) by
\begin{equation}
(\mathcal{C}a)_{k} = \sum_{j > k}\frac{a_{j}}{j},\qquad k\geq 0. \tag{3}
\end{equation}
for a sequence $a=(a_n)_{n\geq 0}\in\ell^2$.
Equivalently,
\[
\mathcal{C}e_{0} = 0,\qquad \mathcal{C}e_{j} = \frac{1}{j}\sum_{k = 0}^{j - 1}e_{k}\quad (j\geq 1).
\]
Its adjoint on \(\ell^2\) is the discrete \emph{Cesàro/Hardy} operator
\begin{equation}
(\mathcal{C}^{*}b)_{0} = 0,\qquad (\mathcal{C}^{*}b)_{j} = \frac{1}{j}\sum_{k = 0}^{j - 1}b_{k}\quad (j\geq 1). \tag{4}
\end{equation}
By the classical discrete Hardy inequality: If \(1< p< \infty\) and \(x = (x_{n})_{n\geq 1}\in \ell^{p}\), then
\begin{equation}
\sum_{n = 1}^{\infty}\left|\frac{1}{n}\sum_{k = 1}^{n}x_{k}\right|^{p}\leq \left(\frac{p}{p - 1}\right)^{p}\sum_{n = 1}^{\infty}|x_{n}|^{p}, \tag{5}
\end{equation}
(see Hardy, Littlewood and Pólya \cite[§9.8, Theorems 326-327]{HLP52}) it follows that
\begin{equation}
\| \mathcal{C}\|_{H^2} = \| \mathcal{C}^{*}\|_{H^2}\leq 2. \tag{6}
\end{equation}
Finally, define a non-unitary diagonal operator \(Q_{\tau}\) by
\[
Q_{\tau}e_{0} = 0,\qquad Q_{\tau}e_{j} = j\big((j + 1)^{i\tau} - j^{i\tau}\big)e_{j},\qquad j\geq 1.
\]
Our desired operator is
\begin{equation}
T_{\tau} = D_{\tau} + \mathcal{C}Q_{\tau} \tag{7}
\end{equation}
This operator is the main focus of the note.

\section{The operator $T_\tau$}

For the boundedness of $T_\tau$, we need the following elementary estimate.

\begin{lemma}\label{lem:2.1}
For every \(j\geq 1\) and \(\tau \in \mathbb{R}\)
\begin{equation}
\left|(j + 1)^{i\tau} - j^{i\tau}\right|\leq |\tau |\log \left(1 + \frac{1}{j}\right)\leq \frac{|\tau|}{j}. \tag{8}
\end{equation}
Consequently,
\begin{equation}
\| Q_{\tau}\| \leq |\tau |. \tag{9}
\end{equation}
\end{lemma}

\begin{proof}
Factoring out \(j^{i\tau}\) whose modulus is one gives
\[
(j + 1)^{i\tau} - j^{i\tau} = j^{i\tau}\left[\exp \left(i\tau \log \left(1 + \frac{1}{j}\right)\right) - 1\right].
\]
Since \(|e^{ix} - 1|\leq |x|\) for real \(x\), the first inequality follows. The second inequality is the elementary bound \(\log (1 + x)\leq x\) for \(x\geq 0\). Multiplying by \(j\) gives (9).
\end{proof}

\begin{proposition}\label{thm:2.2}
For every \(\tau \in \mathbb{R}\), the operator \(T_{\tau}\) defined by (7) is bounded on \(H^{2}\) and
\begin{equation}
\| T_{\tau}\| \leq 1 + 2|\tau |. \tag{10}
\end{equation}
\end{proposition}

\begin{proof}
Since \(D_{\tau}\) is unitary, by (6) and Lemma \ref{lem:2.1} we get
\[
\| \mathcal{C}Q_{\tau}\| \leq \| \mathcal{C}\| \| Q_{\tau}\| \leq 2|\tau |
\]
and therefore
$
\| T_{\tau}\| \leq \| D_{\tau}\| +\| \mathcal{C}Q_{\tau}\| \leq 1 + 2|\tau |.
$
\end{proof}
If \(f(z) = \sum_{j\geq 0}a_{j}z^{j}\in H^2\), then the $k$-th coefficient of $T_\tau f$ is
\begin{equation}
(T_{\tau}f)_{k} = (k + 1)^{i\tau}a_{k} + \sum_{j > k}((j + 1)^{i\tau} - j^{i\tau})a_{j},\qquad k\geq 0. \tag{11}
\end{equation}
Let \(S^{*}\) be the backward shift on \(H^{2}\) and set
\begin{equation}
\Delta := I - S^{*}. \tag{12}
\end{equation}
Thus, if \(f(z) = \sum_{j \geq 0} a_j z^j\), then
\[
(\Delta f)_k = a_k - a_{k + 1}.
\]
Notice that \(\Delta\) is injective on \(H^{2}\) since if \(a_k = a_{k + 1}\) for every \(k\), then \((a_k)\) is constant, and the only constant sequence in \(\ell^2\) is the zero sequence. From (11) one easily sees that
\[
(T_{\tau}f)_{k} - (T_{\tau}f)_{k + 1} = (k + 1)^{i\tau}(a_{k} - a_{k + 1}),
\]
which is precisely the relation
\begin{equation}
\Delta T_{\tau}= D_{\tau}\Delta. \tag{13}
\end{equation}
Hence $T_\tau$ is quasi-conjugate to a diagonal operator and these operators form a group.

\begin{corollary}\label{cor:3.2}
The family \((T_{\tau})_{\tau \in \mathbb{R}}\) is a one-parameter group of bounded invertible operators:
\begin{equation}
T_{\tau}T_{\sigma} = T_{\tau +\sigma},\qquad T_{0} = I,\qquad T_{\tau}^{-1} = T_{-\tau}. \tag{14}
\end{equation}
\end{corollary}

\begin{proof}
Using (13) twice we get
\[
\Delta T_{\tau}T_{\sigma} = D_{\tau}\Delta T_{\sigma} = D_{\tau}D_{\sigma}\Delta = D_{\tau +\sigma}\Delta = \Delta T_{\tau +\sigma}.
\]
Since \(\Delta\) is injective we get \(T_{\tau}T_{\sigma} = T_{\tau +\sigma}\). The rest follow by taking \(\sigma = 0\) and \(\sigma = -\tau\).
\end{proof}

Notice that taking adjoints in (13) gives
\begin{equation}
T_{\tau}^{*}\Delta^{*} = \Delta^{*}D_{-\tau},\qquad \Delta^{*} = I - S \tag{15}
\end{equation}
where \(S\) is the unilateral shift on $H^2$. Since all operators in (7) are bounded, we have
\begin{equation}
T_{\tau}^{*} = D_{-\tau} + Q_{\tau}^{*}C^{*} \tag{16}
\end{equation}
and denoting
$
B_{j}= \sum_{k = 0}^{j}b_{k}$ for $g(z) = \sum_{j = 0}^{\infty}b_{j}e_{j}\in H^2$ allows us to describe the adjoint of $T_\tau$.

\begin{proposition}\label{prop:4.1}
For \(g=\sum_{j = 0}^{\infty}b_{j}e_{j} \in H^2\),
\begin{equation}
(T_{\tau}^{*}g)_{0} = b_{0}, \tag{17}
\end{equation}
and, for every \(j \geq 1\),
\begin{equation}
(T_{\tau}^{*}g)_{j} = (j + 1)^{-i\tau}B_{j} - j^{-i\tau}B_{j - 1}. \tag{18}
\end{equation}
\end{proposition}

\begin{proof}
From (4) and the definition of \(Q_{\tau}\), we have
\[
(T_{\tau}^{*}g)_{j} = (j + 1)^{-i\tau}b_{j} + ((j + 1)^{-i\tau} - j^{-i\tau})B_{j - 1},\qquad j\geq 1.
\]
Since \(b_{j} = B_{j} - B_{j - 1}\), this simplifies immediately to (18). The zeroth coordinate is (17) because \(D_{- \tau}e_{0} = e_{0}\) and \(Q_{\tau}^{*}e_{0} = 0\).
\end{proof}

\section{Vertical translations of kernels $k_s$}
 If \(b_{j}\) denotes the \(j\)th Taylor coefficient of \(k_{s}\), then (B) telescopes to
\begin{equation}
B_{j}:=\sum_{k = 0}^{j}b_{k} = -\frac{(j + 1)^{1 - \bar{s}}}{\bar{s}},\qquad j\geq 0. \tag{19}
\end{equation}
We arrive at the main property of the operators $T_\tau$ which answers the $\mathbf{question}$ inspiring this note.

\begin{proposition}\label{thm:5.1}
Let \(\tau \in \mathbb{R}\) and \(\Re s > \frac{1}{2}\). Then
\begin{equation}
T_{\tau}^{*}k_{s} = \frac{\bar{s} + i\tau}{\bar{s}} k_{s - i\tau}. \tag{20}
\end{equation}
\end{proposition}

\begin{proof}
The zeroth coefficient of \(T_{\tau}^{*}k_{s}\) is  \(- 1 / \bar{s}\) by (17) and (B). For \(j \geq 1\), insert (19) into (18):
\[
(T_{\tau}^{*}k_{s})_{j} = -\frac{1}{\bar{s}}\big((j + 1)^{1 - \bar{s} - i\tau} - j^{1 - \bar{s} - i\tau}\big).
\]
Since \(\overline{s - i\tau} = \bar{s} + i\tau\), multiplication of \(k_{s - i\tau}\) by \((\bar{s} + i\tau) / \bar{s}\) gives exactly these coefficients, including the zeroth one.
\end{proof}

\section{RH as an invariant subspace question}

\begin{theorem}\label{thm:6.1}
Fix \(\tau \in \mathbb{R}\setminus \{0\}\). Then the following are equivalent:

\begin{enumerate}
\item[(i)] the Riemann hypothesis holds;
\item[(ii)] \(\overline{\mathcal{N}}\) is invariant under \(T_{\tau}\);
\item[(iii)] \(T_{\tau}h_{k}\in \overline{\mathcal{N}}\) for some integer \(k\geq 2\).
\end{enumerate}
Moreover, if \(\rho\) is a zero of \(\zeta\) with \(\Re \rho >\frac{1}{2}\) and \(\zeta (\rho - i\tau)\neq 0\), then for every \(k\geq 2\)
\begin{equation}
\mathrm{dist}\big(T_{\tau}h_{k},\overline{\mathcal{N}}\big)\geq \frac{|\zeta(\rho - i\tau)||k^{1 - \rho + i\tau} - 1|}{|\rho|\|k_{\rho}\|} >0. \tag{21}
\end{equation}
\end{theorem}

\begin{proof}
If RH holds, then (A) gives \(\overline{\mathcal{N}} = H^{2}\), so (i) \(\Rightarrow\) (ii). Since \(h_{k}\in \mathcal{N}\) for every \(k\geq 2\), (ii) \(\Rightarrow\) (iii) is immediate. Assume (iii), and suppose that RH is false. By the symmetry of the nontrivial zeros, there is a zero
\[
\rho = \beta +i\gamma ,\qquad \frac{1}{2} < \beta < 1,
\]
with \(\gamma >0\). Then \(k_{\rho}\in \mathcal{N}^{\perp}\) by (D). Since \(T_{\tau}h_{k}\in \overline{\mathcal{N}}\),
\[
0 = \langle T_{\tau}h_{k},k_{\rho}\rangle = \langle h_{k},T_{\tau}^{*}k_{\rho}\rangle .
\]
Using Theorem \ref{thm:5.1} and (C),
\begin{equation}
\langle T_{\tau}h_{k},k_{\rho}\rangle = -\frac{\zeta(\rho - i\tau)}{\rho}\big(k^{1 - \rho + i\tau} - 1\big). \tag{22}
\end{equation}
The second factor is nonzero, because
\[
|k^{1 - \rho +i\tau}| = k^{1 - \beta} > 1.
\]
Hence \(\zeta (\rho - i\tau) = 0\). Repeating the same argument with \(\rho - i\tau\), then \(\rho - 2i\tau\), and so on, gives
\begin{equation}
\zeta (\rho -i n\tau) = 0,\qquad n = 0,1,2,\ldots \tag{23}
\end{equation}
We now invoke Ingham's zero-density estimate. For \(\frac{1}{2} < \sigma < 1\), let
\[
N(\sigma ,T) = \# \{\rho ' = \beta ' + i\gamma ':\zeta (\rho ') = 0,\beta ' \geq \sigma ,0< \gamma ' \leq T\} ,
\]
counting multiplicity. Ingham proved \cite{Ingham40} that
\begin{equation}
N(\sigma ,T)\ll T^{\frac{3(1 - \sigma)}{2 - \sigma}}(\log T)^{5}. \tag{24}
\end{equation}
For every fixed \(\sigma > \frac{1}{2}\), the exponent $\frac{3(1 - \sigma)}{2 - \sigma}$ is strictly less than 1, and therefore
\begin{equation}
N(\sigma ,T) = o(T) \tag{25}
\end{equation}
since $(\log T)^k=o(T^\epsilon)$ for any fixed $k,\epsilon>0$. Choose \(\sigma\) with \(\frac{1}{2} < \sigma < \beta\). If \(\tau < 0\), the zeros in (23) themselves have positive ordinates growing in steps of \(|\tau |\). If \(\tau > 0\), their complex conjugates are zeros and, for all sufficiently large \(n\), have positive ordinates growing in steps of \(\tau\). In either case,
\[
N(\sigma ,T)\geq \frac{T}{|\tau|} +O(1),
\]
contradicting (25). Thus RH holds. It remains to prove the quantitative estimate (21). We use the elementary Hilbert-space bound
\[
g\in \mathcal{N}^{\perp}\quad \Longrightarrow \quad \mathrm{dist}(x,\overline{\mathcal{N}})\geq \frac{|\langle x,g\rangle|}{\|g\|}.
\]
Let \(\rho\) be as in the statement of the theorem. Then \(k_{\rho}\in \mathcal{N}^{\perp}\), so applying the bound with \(x = T_{\tau}h_{k}\) and \(g = k_{\rho}\), and then using (22), yields (21).
\end{proof}

\section{On AI assistance}
In \cite[Prop. 22]{CMSN24}, the common eigenvectors of the adjoints of the $H^2$-operators
\[
W_n f(z) = (1 + z + \cdots + z^{n-1}) f(z^n)
\]
for $n\geq 1$ were characterized. These operators $W_n$ were introduced in \cite{Noor19} and further studied in \cite{ManzurNoorSantos}. In particular 
\[
W_n^* k_s = n^{1-\bar{s}} k_s \quad \forall \, n \geq 1.
\]
So any operator $T$ such that $T^*$ maps $k_s$ ($\Re s>1/2$) to a scalar multiple of $k_{s-i\tau}$ ($\tau\in\mathbb{R}$) must necessarily satisfy the intertwining operator relation
\[
W_n^* T^* = n^{i\tau} T^* W_n^* \ \ \implies \ \ W_n T = n^{i\tau} T W_n.
\]
Then GPT-5.6 Sol was prompted to solve this system along the canonical basis $(z^k)_{k\geq0}$ which resulted in the infinite upper-triangular matrix

\[
T
=
\begin{pmatrix}
1
& 2^{i\tau}-1
& 3^{i\tau}-2^{i\tau}
& 4^{i\tau}-3^{i\tau}
& \cdots
\\[3pt]
0
& 2^{i\tau}
& 3^{i\tau}-2^{i\tau}
& 4^{i\tau}-3^{i\tau}
& \cdots
\\[3pt]
0
& 0
& 3^{i\tau}
& 4^{i\tau}-3^{i\tau}
& \cdots
\\[3pt]
0
& 0
& 0
& 4^{i\tau}
& \cdots
\\[3pt]
\vdots
& \vdots
& \vdots
& \vdots
& \ddots
\end{pmatrix}.
\]

After verifying that this matrix indeed resulted in a bounded operator on $\ell^2$, GPT-5.6 Sol was asked to provide a possible explicit (non-matricial) form for $T$. This resulted in
several possible expressions, but $(7)$ was deemed to be the best for verifying all the required properties directly.

\bigskip

\noindent
\textsc{IMECC, Universidade Estadual de Campinas,
Campinas-SP, Brazil.}

\noindent
\textit{E-mail address:} \texttt{waleed@unicamp.br}
\quad (Waleed Noor, Corresponding Author)

\medskip

\noindent
\textsc{Departamento de Matemáticas y Estadística,
Universidad del Norte, Barranquilla, Colombia.}

\noindent
\textit{E-mail address:} \texttt{juan\_manzur123@hotmail.com}
\quad (Juan Manzur)

\end{document}